\documentclass[11pt]{amsart}

\usepackage{cases}
\usepackage{mathrsfs, euscript}
\usepackage{bbm}
\usepackage{txfonts}
\usepackage{amscd,graphicx}
\usepackage{amsfonts,latexsym,amsmath,amsthm,amsxtra,amssymb,latexsym}

\usepackage{enumerate}
\usepackage{color}
\usepackage{multicol}
\usepackage{hyperref}

\usepackage{todonotes}
\usepackage[marginratio=1:1,height=9.0in,width=6.5in,tmargin=1.0in]{geometry}
\usepackage[normalem]{ulem}
\usepackage{enumerate,enumitem}

    \newtheorem{thm}{Theorem}[section] \newtheorem{cor}[thm]{Corollary}
    \newtheorem{lem}[thm]{Lemma}  \newtheorem{prop}[thm]{Proposition}
    
    \newtheorem {conj}[thm]{Conjecture} \newtheorem{defn}[thm]{Definition}
   \newtheorem {example}[thm]{Example}

    \numberwithin{equation}{section}

\theoremstyle{remark}
\newtheorem*{remark}{Remark}

\newcommand{\la}{\lambda}

\newcommand{\ga}{\gamma}

\newcommand{\floor}[1]{\left\lfloor #1 \right\rfloor}

\DeclareMathOperator{\Vol}{Vol}
\DeclareMathOperator{\MinVol}{\operatorname{MinVol}}

\newcommand{\pa}{\partial}

\newcommand{\op}{\operatorname}
\newcommand{\abs}[1]{\left| #1 \right|}
\newcommand{\norm}[1]{\left\| #1 \right\|}

\newcommand{\set}[1]{\left\{ #1 \right\} }

\newcommand{\grad}{\nabla}
\newcommand{\til}{\widetilde}

\DeclareMathOperator{\Jac}{Jac}

\newcommand{\T}{{\bf T}}

\newcommand{\mc}{\mathcal}
\renewcommand{\bar}{\overline}

\begin{document}

\title{Some remarks on the simplicial volume of nonpositively curved manifolds}

\author{Chris Connell}

\address{Department of Mathematics,
Indiana University, Bloomington, IN 47405, USA}
\email{connell@indiana.edu}

\author{Shi Wang}
\address{Max Planck Institute for Mathematics, Bonn, Germany}
\email{shiwang.math@gmail.com}


\date{}

\begin{abstract}
We show that any closed manifold with a metric of nonpositive curvature that admits either a single point rank condition or a single point curvature condition has positive simplicial volume. We use this to provide a differential geometric proof of a conjecture of Gromov in dimension three.
  \end{abstract}

\maketitle

\section{Introduction and Results}

The vector space $C_i(X; \mathbb R)$ of singular $i$-chains of a topological
space $X$ comes equipped with a natural choice of basis consisting of the set of
all continuous maps from the $i$-dimensional Euclidean simplex into $X$. The
$\ell^1$-norm on $C_i(X; \mathbb R)$ associated to this basis descends to a
semi-norm $\norm{\cdot }_1$ on the singular homology $H_i(X; \mathbb R)$ by
taking the infimum of the norm within each equivalence class. The {\em
simplicial volume,} written $\norm{M}$, of a closed oriented $n$-manifold $M$ is
defined to be $\norm{[M]}_1$, where $[M]$ is the fundamental class in $H_n(M;
\mathbb R)$. More concretely,
\[
\norm{M}=\inf \set{\sum_i \abs{a_i}\,:\, \left[\sum_i a_i \sigma_i\right]=[M]\in H_n(M,\mathbb R) },
\]
where the infimum is taken over all singular cycles with real coefficients
representing the fundamental class in the top homology group of $M$.  This invariant is multiplicative under covers, so the definition can be extended to closed non-orientable manifolds as well.

This invariant was first introduced by Thurston (\cite[Chapter
6]{Thurston77}) and soon after expanded upon by  Gromov (\cite{Gromov82}). This topological invariant measures how efficiently the fundamental class of $M$ can be represented by real singular cycles. The simplicial volume relates to other important geometrically defined topological invariants such as the {\em minimal volume} invariant defined by,
\[
\MinVol(M)=\set{\Vol(M,g)\,:\, -1\leq K_g\leq 1}.
\]
Here $K_g$ denotes the sectional curvatures of the metric $g$. We always have $\MinVol(M)\geq C\norm{M}$ for some universal constant $C>0$ depending only on dimension (\cite{Gromov82}).

The minimal volume invariant, and hence simplicial volume, of $M$ vanishes if $M$ admits a nondegenerate circle action, or more generally a polarized $\mc{F}$-structure (\cite{Fukaya87a,Cheeger86a,Cheeger90}). On the other hand, positivity of simplicial volume has been established for only a few special classes of manifolds including higher genus surface-by-surface bundles (\cite{Hoster01}), negatively curved manifolds (\cite{Gromov82}), more generally manifolds with hyperbolic fundamental group (\cite{Mineyev01}) including visibility manifolds of nonpositive curvature (\cite{Cao95}), certain generalized graph manifolds (\cite{Kuessner04,Connell17}), and most closed locally symmetric spaces of higher rank  and some noncompact finite volume ones as well (\cite{Lafont06b,Bucher-Karlsson07a,Kim12,Loeh09}). 

The relationship between curvature and simplicial volume is delicate. Gromov \cite{Gromov82} showed that closed manifolds $M$ with amenable fundamental group have $\norm{M}=0$. This includes manifolds with finite fundamental groups such as those with positive Ricci curvatures. Moreover, Cheeger and Gromoll \cite{Cheeger-Gromoll} showed that any manifold $M$ of nonnegative Ricci curvature has fundamental group which is a finite extension of a crystallographic group. In particular, $\pi_1(M)$ is amenable and $\norm{M}=0$. On the other hand, Lohkamp \cite{Lokhamp94} showed that any manifold of dimension at least three admits a metric of negative Ricci curvature, and hence there can be no relation to simplicial volume. Restricting our attention to nonpositively curved manifolds, Gromov (\cite{Savage82} and see also \cite[p.11]{Gromov82}) conjectured the following:
\begin{conj}[Gromov]\label{conj:Gromov}
Any closed manifold of nonpositive sectional curvature and negative Ricci curvature has $\norm{M}>0$.
\end{conj}

For the statement of next theorem we need the following definition.
\begin{defn}\label{def:k-th trace} For any positive (resp. negative) semi-definite linear endomorphism $A:\:V\rightarrow V$ of an $n$-dimensional vector space $V$,
and for any $k=1,2,...,n$, we define the $k$-th trace of $A$, denoted by
$\op{Tr}_k(A)$, to be the sum of the $k$ eigenvalues of $A$ closest to $0$.

In other words,
\[
\op{Tr}_k(A)=\inf_{V_k\subset V}\op{Tr}(A|_{V_k})\quad\left(\text{resp. }\op{Tr}_k(A)=\sup_{V_k\subset V}\op{Tr}(A|_{V_k})\right),
\]
where $V_k$ is a $k$-dimensional subspace (not necessarily invariant under $A$) of $V$, and $A|_{V_k}$ is the restriction of $A$ to $V_k\times V_k$, viewed as a bilinear form. Hence $\op{Tr}(A|_{V_k})=\sum_{i=1}^k A(e_i,e_i)$ for any basis of $V_k$. 
\end{defn}

In what follows we denote the covariant Hessian of the Busemann function determined by $v\in T^1M$ by $DdB_v$. 

\begin{thm}\label{thm:general estimate} If $M$ is a closed nonpositively curved manifold of dimension $n\geq 3$, then the simplicial volume has a lower bound
$$||M||\geq \frac{2(n-1)^n}{n^{n/2}\omega_n}\int_M u^n(x)dV$$
where $\omega_n$ is the volume of the unit round $n$-sphere, and $u(x)=\inf_{v\in T_x^1M}\op{Tr}_2DdB_v$.
\end{thm}

\begin{remark} Note that the theorem holds trivially when $u(x)\equiv 0$, and $u(x)>0$ if and only if every vector in $T_x^1M$ is $\op{rank}^+$ one (see Definition \ref{def:rank} and Lemma \ref{lem:tr2-implies-rank1} below).
\end{remark}

\begin{example}
	\begin{figure}[ht]
		\centering
		\includegraphics[width=0.7\linewidth]{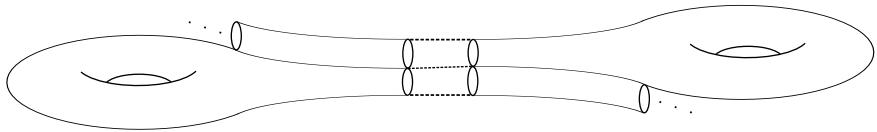}
		\caption{Gromov's example of a rank one graph manifold.}
		\label{fig:gromovs-example}
	\end{figure}
The example of Figure \ref{fig:gromovs-example} consists of two hyperbolic surfaces $\Sigma_1$ and $\Sigma_2$ with one puncture each such that the cusps have been truncated and smoothly and symmetrically tapered to a flat metric in a neighborhood of their round circle boundaries $\pa \Sigma_i$. We form $M$ by gluing $\Sigma_1\times \pa \Sigma_2$ to $\pa \Sigma_1 \times \Sigma_2$ by the identity isometry along the flat boundary torus $\T^2=\pa \Sigma_1 \times \pa \Sigma_2$, thus switching the surface and circle factors. 

This manifold has an obvious rank one polarized $\mc{F}$-structure formed by the local circle bundles, and hence $\norm{M}=0$. In this case, $u(x)=0$ for all $x$. Indeed, if $v$ is tangent to the circle direction at a point $x$, then $DdB_v(w)=0$ for all $w\in v^\perp$ and hence $\op{Tr}_2 DdB_v=0$. 
\end{example}

\begin{cor}\label{cor:positive-under-rank-condition} Let $M$ be a closed nonpositively curved manifold. If there exists a point $x\in M$ so that every vector in $T_x^1M$ is $\op{rank}^+$ one, then the simplicial volume $||M||>0$.
\end{cor}

\begin{cor} Let $M$ be a closed nonpositively curved manifold. If there exists a point $x\in M$ so that every 2-plane in $T_xM$ is negatively curved, then the simplicial volume $||M||>0$.
\end{cor}

\begin{remark}
The $n\geq 3$ cases of the above corollaries are immediate consequences of Thereom \ref{thm:general estimate}, and the $n=2$ case holds since the hypotheses on the metric rule out the torus, Klein-bottle, projective plane and sphere.
\end{remark}

In the special case where $M$ is real hyperbolic, $u(x)\equiv 1$ and the simplicial volume $||M||=Vol(M)/\sigma_n$ (\cite{Gromov82, Thurston77}), where $\sigma_n$ is the maximal volume of ideal simplices in $\mathbb H^n$ . 

\begin{remark} 
The above observation immediately yields from Theorem \ref{thm:general estimate} the upper bound,
$$\sigma_n\leq \frac{n^{n/2}\omega_n}{2(n-1)^n}.$$
While this is inferior to the upper bound of $\frac{\pi}{(n-1)!}$ given by Thurston (\cite{Thurston77}, see also \cite{Haagerup-Munkholm}), we achieve it via a completely indirect computation. 
\end{remark}

Using Theorem \ref{thm:general estimate} we provide a purely differential geometric proof of the following 3-dimensional case of Conjecture \ref{conj:Gromov}.
\begin{thm}\label{thm:dim3}

 If $M^3$ admits a nonpostively curved metric with negative Ricci curvature, then the simplicial volume $||M||>0$.

\end{thm}

For the next theorem we need to introduce the notion of $k$-Ricci curvature.
\begin{defn}\label{def:k-ricci}
	For $M$ nonpositively curved and $u,v\in T_xM$, the $k$-Ricci curvature is given by
	\[
	\op{Ric}_{k}(u,v)=\sup_{\stackrel{V\subset T_xM}{\dim V=k}} \op{Tr}
	R(u,\cdot,v,\cdot)|_V=\op{Tr}_k R(u,\cdot,v,\cdot).
	\]
\end{defn}

The next theorem generalizes Theorem 1 of \cite{CW17}. 
\begin{thm}\label{thm:pointwise-k-ricci}
	Let $M$ be an closed manifold of dimension $n$ admitting a
	Riemannian metric of nonpositive curvature. If there exists $x\in M$, such that any vector $v_x\in T^1_xM$ satisfies $\op{Ric}_{\floor{\frac{n}{4}}+1}(v_x, v_x) < 0$, then the
	simplicial volume $||M||> 0$.
\end{thm}

\begin{remark}
Dual to the aforementioned $\ell^1$-norm on $C_*(X; \mathbb R)$, we also have an
$\ell^\infty$-norm $\norm{\cdot }_\infty$ on the real vector spaces $C^*(X;
\mathbb R)= \text{Hom}_{\mathbb R}(C_*(X; \mathbb R) , \mathbb R) = [C_*(X;
\mathbb R)]^{\vee}$ appearing in the cochain complex for singular cohomology
$H^*(X; \mathbb R)$. By considering the bounded elements, one obtains a
subcomplex of the cochain complex, whose homology yields the {\it bounded
cohomology} $H^*_b(X; \mathbb R)$. The natural inclusion of cochain complexes
induces a comparison map $c: H^*_b(X; \mathbb R) \rightarrow H^*(X; \mathbb R)$
from the bounded cohomology to the ordinary cohomology. Elements in the image
are cohomology classes which have bounded representatives, and so admit a
well-defined $\ell^\infty$-norm. Hence, the simplicial volume of $M$ vanishes if
and only if the comparison map $c:H^n_b(M,\mathbb R) \rightarrow H^n(M,\mathbb R)$ in top dimension is the zero map. In the setting of the above theorems, positivity of the simplicial volume implies that the top dimentional comparison map is surjecive and in particular, the bounded cohomology $H^n_b(M,\mathbb R)$ is nontrivial.

Note however, that the pointwise estimate of the above theorem is not necessarily sufficient to give surjectivity of the comparison map for bounded cohomology in lower dimensions analogous to the statement of Theorem 2 in \cite{CW17}.
\end{remark}

\subsection*{Acknowledgments} The authors would like to thank Werner Ballmann and Jean-Fran\c{c}ois Lafont for helpful discussions and the anonymous referees for both a careful scrutiny and suggesting several helpful clarifications.

\section{Barycenters and Straightening}\label{sec:PS}

\subsection{Rank of a manifold} In the context of nonpositively curved manifold, the notion of geometric rank can be viewed as a generalization of the usual rank of a symmetric space. We recall the following definition,

\begin{defn}\label{def:rank}
	Let $M$ be a nonpositively curved manifold. For any nonzero vector $v\in TM$, we define the {\em rank} of $v$, denoted $\op{rank}(v)$ to be the dimension of the space of all parallel Jocobi fields along the geodesic tangent to $v$. We say $M$ is {\em rank one} if there exists a rank one vector on $TM$, and is {\em higher rank} otherwise. Similarly we define $\op{rank}^+(v)$ to be the dimension of the space of all parallel Jocobi fields along the geodesic ray tangent to $v$.
\end{defn}

\begin{remark}
$\op{rank}^+$ one vectors are automatically rank one. For backward recurrent vectors $v$ we have $\op{rank}^+(v)=\op{rank}(v)$. When $M$ is compact, backward recurrent vectors are dense (\cite{Ballmann-Brin-Eberlein85}). Since $\op{rank}^+$ and $\op{rank}$ are upper semicontinuous on $T^1M$, it follows that a compact manifold is $\op{rank}$ one if and only if it is $\op{rank}^+$ one.
\end{remark}

One remarkable result about higher rank manifolds is that they are well understood, and completely classified by the following higher rank rigidity theorem,

\begin{thm}\label{thm:rank-rigidity}(\cite{Ballmann85,BS87})
	If $M$ is closed, nonpositively curved manifold of higher rank, then the universal cover $\til M$ is either a nontrivial product of nonpostively curved manifolds, or a higher rank symmetric space of noncompact type.
\end{thm}

For our purpose, we would like to see how the rank is reflected by the conditions of our theorems, in terms of $u(x)$ and $\op{Ric}_k$. In fact, we will see in Lemma \ref{lem:tr2-implies-rank1} that if $u(x)$ is not identically zero on a closed manifold $M$, then $M$ is rank one. On the other hand, we have

\begin{lem}\label{lem:k-ricci-implies-rank1}
If there exists a point $x\in M$ such that $\op{Ric}_{\floor{\frac{n}{4}}+1}(v_x, v_x) < 0$ for all $v_x\in T^1_xM$, then $M$ is rank one.
\end{lem} 
\begin{proof} We begin to show that the universal cover $\til M$ cannot be a product. If $\til M$ was a product, we can write $\til M=M_1^{(n_1)}\times M_2^{(n_2)}$, where $n_1\leq n/2\leq n_2$. Any vector $v\in T^1M_1\oplus 0\subset T^1 \til M$ has $\op{Ric}_{n_2}(v,v)=0$ since all vectors in the $M_2$ factor have zero sectional curvature with $v$. This implies $\op{Ric}_{\floor{\frac{n}{4}}+1}(v,v)=0$ as $\floor{\frac{n}{4}}+1\leq n/2\leq n_2$, contrary to the hypothesis. 

Secondly, we show $\til M$ cannot be an irreducible symmetric space either. If not, we define the splitting rank of a symmetric space $\til M$, denoted by $\op{srk}(\til M)$, to be the maximal dimension of a totally geodesic submanifold $Y\subset \til M$ which splits off an isometric $\mathbb R$-factor. If we take a vector $v$ corresponding to the $\mathbb R$-factor of a submanifold $Y$ that attains the splitting rank, all vectors in $T^1Y\subset T^1\til M$ will have zero sectional curvatures with $v$, hence $\op{Ric}_{\op{srk}(X)}(v,v)=0$. On the other hand, the splitting rank has been computed explicitly in Table 1 of \cite{Wang16} for all irreducible symmetric spaces of noncompact type, and in particular, we have $\floor{\frac{n}{4}}+1\leq\op{srk}(\til M)$. Therefore, we have $\op{Ric}_{\floor{\frac{n}{4}}+1}(v,v)=0$, giving a contradiction.

If $\til M$ is neither a product nor symmetric, then $M$ is rank one by Theorem \ref{thm:rank-rigidity}.
\end{proof}

\subsection{Busemann functions}

Recall that, for any triple $(x,y,\theta)\in \til M\times \til M\times \partial \til M$, the Busemann function $B$ on $\til M$ is defined by
$$B(x,y,\theta)=\lim_{t\rightarrow \infty}(d_{\til{M}}(y,\gamma_\theta(t))-t)$$
where $\gamma_\theta(t)$ is the unique geodesic ray from $x$ to $\theta$. Denote by $v_{x,\theta}\in T^1_xM$ the unit vector $v_{x,\theta}=\gamma_\theta'(0)$.

We fix a basepoint $O$ in $\til{M}$ and shorten $B(O,y,\theta)$ to just $B(y,\theta)$.
We note that for fixed $\theta\in \partial\til{M}$ the Busemann function $B(x,\theta)$ is
convex on $\til{M}$, and the nullspace of its Hessian $DdB_{(x,\theta)}$ consists of the initial vectors of all parallel Jacobi fields along the geodesic ray tangent to $v_{x,\theta}$ (see the proof of Lemma \ref{lem:tr2-implies-rank1}). In particular, we have the following lemma:

\begin{lem}\label{lem:tr2-implies-rank1}
$v_{x,\theta}$ is $\op{rank}^+$ one if and only if $\op{Tr}_2DdB_{(x,\theta)}>0$. As a result, if there exists a point $p\in M$, so that $u(p)>0$, then $M$ is rank one. 
\end{lem}

\begin{proof}
For any $w\in v_{x,\theta}^\perp\subset T_xM$, $DdB_{(x,\theta)}(w)=0$ if and only if there is a stable Jacobi field $J$ along the geodesic tangent to $v_{x,\theta}$ with $J(0)=w$ and $J'(0)=0$ (Proposition 3.1 of \cite{Heintze-ImHof}). By convexity of the norm of Jacobi fields in nonpositively curved manifolds, $\norm{J(t)}$ is constant in $t$ and so $J(t)$ is simply the parallel translation of $w$ along $g^t v_{x,\theta}$. In particular, the existence of such a field implies $\op{rank}^+(v_{x,\theta})\geq 2$.
\end{proof}

\subsection{Lyapunov exponents}
For any $v\in T^1M$ define the \emph{lower Lyapunov exponent} at $v$ to be the quantity
\[
\la^-_v=\min_{u\in v^\perp\cap T^1M}\bar{\lim_{t\to\infty}} \frac{1}{t}\log \norm{\Lambda_{v,t}(u)}, 
\]
where $\Lambda_{v,t}:v^\perp\to (g^tv)^\perp$ is the unstable Jacobi tensor along the geodesic through $v$. Consequently, $\op{Tr}_2DdB_v=0$ implies $\la^-_v= 0$.

For $v\in T^1M$ we define the (weak) stable manifold through $v$ to be
\[
W^s(v)=\set{w\in T^1\til{M}\,:\, d_{T^1\til{M}}(g^tw,g^tv)\leq C \text{ for all $t\geq 0$ and some $C>0$}}.
\] 
The following is a restatement of Proposition 3.1 of \cite{Connell03b} in our setting. (In that paper $\la^-_v$ is written $\la^+_v$ since it coincides with the negative of the largest nonpositive stable Lyapunov exponent.)
\begin{lem}\label{lem:la-boundary}
The quantity $\la^-_v$ is constant on the stable manifold $W^s(v)$, and consequently only depends on the point $v(\infty)\in \partial \til{M}$. 
\end{lem}

We will thus define $\la^-_\theta$ to be $\la^-_{v_{x,\theta}}$ for any $x\in \til{M}$.

\subsection{Patterson-Sullivan measures and barycenters}

Let $M$ be a closed nonpositively curved rank one
manifold, $\til{M}$ the universal cover of $M$, and $\Gamma$ the fundamental group of $M$. In \cite{Knieper97}, Knieper showed that there exists a unique (up to scale) family of finite Borel measures $\{\mu_x\}_{x\in \til{M}}$ fully supported on $\partial \til{M}$, called the Patterson-Sullivan measures, which satisfies:
\begin{enumerate}
	\item $\mu_x$ is $\Gamma$-equivariant, i.e. $\ga_*\mu_{x}=\mu_{\gamma x}$ for all $x\in \til{M}$ and $\gamma\in \Gamma$,  and
	\item $\frac{d\mu_x}{d\mu_y}(\theta)=e^{hB(x,y,\theta)}$, for all $x,y\in \til{M}$, and
	$\theta\in \partial\til{M}$.
\end{enumerate}
where $h$ is the volume entropy of $M$, and $B(x,y,\theta)$ is the Busemann function of
$\til{M}$. 

If $\nu$ is any finite Borel measure fully supported on
$\partial\til{M}$, by taking the integral of $B(x,\theta)$ with respect to $\nu$, we
obtain a function
$$x\mapsto \mathcal{B}_{\nu}(x):=\int_{\partial \til{M}}B(x,\theta)d\nu(\theta)$$

\begin{lem}\label{lem:strictly convex} If $M$ is rank one and $\nu$ is any finite Borel measure
that is fully supported on $\partial\til{M}$, then the function $x\mapsto \mathcal{B}_{\nu}(x)$ is strictly convex. Moreover, if $\nu$ has no atoms and $\nu\left(\set{\theta\in\partial\til{M}: \la^-_\theta=0}\right)=0$, then this map attains a unique minimum, denoted $\op{bar}(\nu)$, in $\til{M}$.
\end{lem}

\begin{proof} For the first statement, it is sufficient to show that the Hessian
	$$\int_{\partial\til{M}}DdB_{(x,\theta)}(\cdot,\cdot)d\nu(\theta)$$
	is positive definite for all $x\in\til{M}$. Suppose, to the contrary, that at some point $x$ this operator has a nonzero element $w\in T_xM$ in its kernel. Since $DdB_{x,\theta}$ is nonnegative and $\nu$ is fully supported, $DdB_{(x,\theta)}(w)=0$ for almost all $\theta\in \partial\til{M}$, and hence by continuity all $\theta \in \partial\til{M}$.  In other words, $\op{Tr}_2DdB_{(x,\theta)}=0$ for all $\theta\in \partial\til{M}$, except possible when $v_{x,\theta}=\pm w$. However,  by continuity  $\op{Tr}_2DdB_{(x,\theta)}=0$ for all $\theta\in \partial\til{M}$. By Lemma \ref{lem:tr2-implies-rank1}, $v_{x,\theta}$ is not $\op{rank}^+$ one for any $\theta\in \partial\til{M}$. 
	
	Since $M$ is rank one,  Corollary 1.2 of \cite{Knieper98} implies that there is a dense set of periodic rank one vectors in $T^1M$. If $v_{y,\theta}$ is such a vector, then $\la^{-}_{\theta}>0$. By Lemma \ref{lem:la-boundary}, $v_{x,\theta}$ is also $\op{rank}^+$ one, a contradiction.
	
	For the last assertion, observe that $\grad B(x,\theta)=-v_{x,\theta}$ and hence 
	\[
	-\grad\mathcal{B}_{\nu}(x)=\int_{\partial\til{M}}v_{(x,\theta)}d\nu(\theta).
	\] 
	Let $\ga:[0,\infty)\to \til{M}$ be any geodesic ray and let $\theta\in \partial\til{M}$ with $\theta\neq \ga(\infty)$ be such that $\la^{-}_\theta>0$. By Lemma \ref{lem:tr2-implies-rank1} the horoballs tangent to $\theta$, i.e. the sub-level sets of $B(\cdot,\theta)$, are strictly convex. Hence if $H_t$ is the unique horoball at $\theta$ whose boundary horosphere passes through $\ga(t)$, then for $t$ sufficiently large $\ga(t+s)$ is outside $H_t$ for all $s>0$. Since $v_{\ga(t),\theta}$ is the inward pointing normal to $\partial H_t$ at $\ga(t)$, we have $\angle_{\ga(t)}(\ga'(t),v_{\ga(t),\theta}) > \frac{\pi}{2}$. Since by hypothesis $\nu$ has no atoms and $\nu\left(\set{\theta\in\partial\til{M}: \la^-_\theta=0}\right)=0$, we have $\nu\left(\set{\theta\,:\, \angle_{\ga(t)}(\ga'(t),v_{\ga(t),\theta})<\frac{\pi}{2} }\right)$ tends to $0$ as $t\to\infty$. Note that $\ga$ may intersect the boundary of $H_t$ in at most two points $\set{\ga(t'),\ga(t)}$. When there are two points, suppose $t>t'$. Thus in all cases,  $-\grad\mathcal{B}_{\nu}(\ga(t))$ makes an angle greater than $\frac{\pi}{2}$ with $\ga'(t)$. Hence for any $p\in \til{M}$, by compactness of the sphere $T^1_p\til{M}$ and continuity of $-\grad\mathcal{B}_{\nu}(\ga(t))$ there is an $R>0$ such that on the geodesic sphere $S(p,R)$, $-\grad\mathcal{B}_{\nu}(\ga(t))$ belongs to the half-space of the inward pointing normal. Thus this gradient field attains its unique zero in $B(p,R)$, which is the point $\op{bar}(\nu)$ where the minimum of $\mathcal{B}_{\nu}$ occurs.  
\end{proof}

We note that $\op{bar}(\nu)$ does not depend on the choice of basepoint $O$, since changing $O$ only changes each Busemann function, and hence $\mathcal{B}_\nu$, by an additive constant. 
 
\subsection{Straightening subordinated to $U$}

\begin{defn}\label{def:straightening} (Compare \cite{Lafont06b}) Let $\til{M}^n$ be the universal
cover of an $n$-dimensional manifold $M^n$, and $p: \til{M}^n\rightarrow M^n$ be the covering map. $U\subset M^n$ is a nonempty open set. We denote by $\Gamma$ the fundamental group
of $M^n$, and by $C_{\ast}(\til{M}^n)$ the real singular chain complex of
$\til{M}^n$. Equivalently, $C_{k}(\til{M}^n)$ is the free $\mathbb{R}$-module
generated by $C^0(\Delta^k,\til{M}^n)$, the set of singular $k$-simplices in
$\til{M}^n$, where $\Delta^k$ is equipped with some fixed Riemannian metric. We say a
collection of maps $st_k:C^0(\Delta^k,\til{M}^n)\rightarrow C^0(\Delta^k,\til{M}^n)$
is a straightening subordinated to $U$, if it satisfies the following conditions:
\begin{enumerate}
	\item the maps $st_k$ are $\Gamma$-equivariant,
	\item the maps $st_{\ast}$ induce a chain map $st_{\ast}:C_{\ast}(\til{M}^n,\mathbb
	R)\rightarrow C_{\ast}(\til{M}^n,\mathbb R)$ that is $\Gamma$-equivariantly chain
	homotopic to the identity,
	\item the image of $st_n$ lies in $C^1(\Delta^n, \til{M}^n)$, that is, the top
	dimensional straightened simplices are $C^1$,
	\item there exists a constant $C$ depending on $\til{M}^n$, $U$, and the chosen
	Riemannian metric on $\Delta^n$, such that for any pair $(f,\delta)\in C^0(\Delta^n, \til{M}^n)\times \Delta^n$ satisfying $st_n(f)(\delta)\in p^{-1}(U)$, we have a uniform upper bound on the Jacobian of $st_n(f)$ at $\delta$:
	$$|\op{Jac}(st_n(f))(\delta)|\leq C$$
where  $st_n(f):\Delta^n\rightarrow\til{M}^n$ is the corresponding straightened simplex of $f$.
\end{enumerate}
\end{defn}

\begin{thm}\label{thm:straightening}  If $\til{M}^n$ admits a straightening subordinated to some nonempty open set $U$, then the
simplicial volume of $M$ is positive.
\end{thm}
\begin{proof} We choose a nontrivial smooth bump function $\phi(x)$ on $M$, such that $0\leq \phi \leq 1$, and $\phi(x)=0$ for all $x\notin U$. Let $\sum_{i=1}^la_i\sigma_i$ be a singular chain representing the fundamental class $[M]$ in the real coefficient, and $st(\sigma_i)$ be the straightened simplex of $\sigma_i$ on $M$, with lift $\til{st (\sigma_i)}$ on the universal cover $\til M$. We have
\begin{align}\label{eq:str-htp-equiv}
\int_M\phi(x)dV &=\int_{[\sum a_i\sigma_i]}\phi(x)dV =\int_{[\sum a_ist(\sigma_i)]}\phi(x)dV \\ \label{ineq:lift}
& \leq \sum_{i=1}^l |a_i|\cdot\left|\int_{\til{ st(\sigma_i) }  }\til{\phi}(x)d\til V\right|\\ \label{eq:U supp phi}
& =  \sum_{i=1}^l |a_i|\cdot\left|\int_{\til{st(\sigma_i)}\cap p^{-1}(U)}\til{\phi}(x)d\til V\right|\\ \label{ineq:pull-back}
& \leq \sum_{i=1}^l |a_i|\cdot\left|\int_{(p\circ \til{st(\sigma_i)})^{-1}(U)} \phi( st(\sigma_i)(\delta))|Jac(st(\sigma_i))(\delta)|dV_{\Delta}\right|\\ \label{ineq:Jac-control}
& \leq  \sum_{i=1}^l |a_i|\cdot C\op{Vol}(\Delta^n)
\end{align}
where equation (\ref{eq:str-htp-equiv}) follows from (b) of Definition \ref{def:straightening}, inequality (\ref{ineq:lift}) lifts to the universal cover $\til M$, equation (\ref{eq:U supp phi}) uses the support of $\phi$, inequality (\ref{ineq:pull-back}) pulls the integral back on $\Delta^n$, and inequality (\ref{ineq:Jac-control}) follows from (c)(d) of  Definition \ref{def:straightening}.

By taking the infimum over all fundamental class representatives $\sum a_i\sigma_i$, we have
$$||M||\geq \frac{\int_M\phi(x)dV}{C\op{Vol}(\Delta^n)}>0$$

\end{proof}

\subsection{Barycentric straightening}
The barycentric straightening was introduced by Lafont and Schmidt \cite{Lafont06b}
(based on the barycenter method originally developed by Besson, Courtois, and Gallot
\cite{Besson95}) to show the positivity of simplicial volume of most locally symmetric spaces of noncompact type. 

Briefly speaking, any old $k$-simplex on a manifold gives $k+1$ vertices, which forms $k+1$ Patterson-Sullivan measures. These measures can be viewed as $k+1$ vertices in the affine space of all measures supported on the boundary of the manifold. Using these vertices, we can fill up a simplex on the space of measures by linear combinations. Finally, applying the barycenter map, it gives a new simplex on the original manifold.

More explicitly, we denote by
$\Delta^k_s$ the standard spherical k-simplex in the Euclidean space, that is
$$\Delta^k_s=\Big\{(a_1,\ldots ,a_{k+1})\mid a_i\geq 0,
\sum_{i=1}^{k+1}a_i^2=1\Big\}\subseteq \mathbb{R}^{k+1},$$
with the induced Riemannian metric from $\mathbb{R}^{k+1}$, and with ordered vertices
$(e_1,\ldots,e_{k+1})$. Given any singular $k$-simplex $f:\Delta^k_s\rightarrow
\til{M}$, with ordered vertices
$V=(x_1,\ldots,x_{k+1})=\left(f(e_1),\ldots,f(e_{k+1})\right)$, we define the
$k$-straightened simplex
$$st_k(f):\Delta^k_s\rightarrow \til{M}$$
$$st_k(f)(a_1,\ldots ,a_{k+1}):=\op{bar}\left(\sum_{i=1}^{k+1}a_i^2\nu_{x_i}\right)$$
where $\nu_{x_i}=\mu_{x_i}/\|\mu_{x_i}\|$ is the normalized Patterson-Sullivan measure at
$x_i$. We notice that $st_k(f)$ is determined by the (ordered) vertex set $V$, and we
denote $st_k(f)(\delta)$ by $st_V(\delta)$, for $\delta\in\Delta^k_s$.

When $M$ is rank one, the Patterson-Sullivan measures $\nu_{x_i}$ are fully supported on $\partial \til{M}$, and so are the linear combinations. By Lemma 4.4 of \cite{Knieper97} these measures have no atoms. By Lemma 2.4 and ensuing remarks of \cite{Knieper98} these measures are the boundary conditionals of the the measure $\mu$ of maximal entropy on $T^1M$. By Theorem 1.1 of \cite{Knieper98}, $\mu$ is supported on the rank one vectors. Since $\mu$ is ergodic, the uniformly recurrent rank one vectors have full $\mu$-measure. By the corresponding unstable version of Lemma 1.4 of Appendix 1 of \cite{Ballmann-Gromov-Schroeder} any uniformly recurrent vector $v$ has $\la^{-}_v>0$. Therefore we must have $\nu_x\left(\set{\theta\,:\, \la^-_\theta=0}\right)=0$. Hence by Lemma \ref{lem:strictly convex}, the barycenter map is well-defined, and so is the barycentric straightening. In order to show the positivity of the simplicial volume, we need to check the barycentric straightening is a straightening subordinated to some non-empty open set in the sense of Definition \ref{def:straightening}.

It is easy to see the barycentric straightening satisfies (a)-(c) of Definition \ref{def:straightening}. The proofs of these three properties are established  in \cite[Property (1)-(3)]{Lafont06b}, where both the proofs can be followed verbatim once their Patterson-Sullivan measures $\nu(x)$ integrated on the Furstenberg boundary is changed to our $\nu_x$ integrated on the entire ideal boundary. While that paper is concerned with the higher rank case, the proofs of those properties do not use that fact and work in the context of any Hadamard manifold once the above change is made and provided the barycenter exists and $\mc{B}_\nu$ is strictly convex, which we have already established for our family of measures.

To check (d) that the Jacobian of top dimensional straightened simplices are
uniformly bounded on some non-empty open set $U\subset M$, we make the following
estimate(which is also similar to \cite[Property (4)]{Lafont06b}).

For any $\delta=\sum_{i=1}^{n+1}a_ie_i\in \Delta_s^{n}$, $st_n(f)(\delta)$ is defined to be
the unique point where the function
$$x\mapsto \int_{\partial
	\til{M}}B(x,\theta)d\left(\sum_{i=1}^{n+1}a_i^2\nu_{x_i}\right)(\theta)$$
is minimized. Hence, by differentiating at that point, we get the 1-form equation
$$\int_{\partial\til{M}}dB_{(st_V(\delta),\theta)}(\cdot)d\left(\sum_{i=1}^{n+1}a_i^2\nu_{x_i}\right)(\theta)\equiv
0$$
which holds identically on the tangent space $T_{st_V(\delta)}\til{M}$.
Differentiating in a direction $u\in T_\delta(\Delta_s^n)$ in the source, one obtains the $2$-form equation

\begin{align}\label{eqn:2-form}
\begin{split}
&\quad\sum_{i=1}^{n+1}2a_i\langle
u,e_i\rangle_\delta\int_{\partial\til{M}}dB_{(st_V(\delta),\theta)}(w)d\nu_{x_i}(\theta)\\
&+\int_{\partial\til{M}}DdB_{(st_V(\delta),\theta)}(D_\delta(st_V)(u),w)d\left(\sum_{i=1}^{n+1}a_i^2\nu_{x_i}\right)(\theta)\equiv 0
\end{split}
\end{align}
which holds for every $u\in T_\delta(\Delta_n^k)$ and $w\in T_{st_V(\delta)}(\til{M})$.

Now we define two positive semidefinite symmetric endomorphisms $H_{\delta,V}$ and $K_{\delta,V}$
on $T_{st_V(\delta)}(\til{M})$:
$$\langle
H_{\delta,V}(w),w\rangle_{st_V(\delta)}=\int_{\partial\til{M}}\Big(dB_{(st_V(\delta),\theta)}(w)\Big)^2 d\left(\sum_{i=1}^{n+1}a_i^2\nu_{x_i}\right)(\theta)$$
$$\langle
K_{\delta,V}(w),w\rangle_{st_V(\delta)}=\int_{\partial\til{M}}DdB_{(st_V(\delta),\theta)}(w,w)d\left(\sum_{i=1}^{n+1}a_i^2\nu_{x_i}\right)(\theta)$$
Note from Lemma \ref{lem:strictly convex} that $K_{\delta,V}$ is positive definite for any vertex set $V$. From
Equation (\ref{eqn:2-form}), we obtain, for $u\in T_\delta (\Delta _s^n)$ a unit vector
and $w\in T_{st_V(\delta)}(\til{M})$ arbitrary, the following
\begin{equation}\label{eqn:Q1-bounds-Q2}
\begin{split}
|\langle K_{\delta,V}\Big(D_\delta(st_V)(u)\Big),w\rangle|&=|-\sum_{i=1}^{n+1}2a_i\langle
u,e_i\rangle_\delta\int_{\partial\til{M}}dB_{(st_V(\delta),\theta)}(w)d\nu_{x_i}(\theta)|
\\
& \leq \left(\sum_{i=1}^{n+1}\langle
u,e_i\rangle_\delta^2\right)^{1/2}\left(\sum_{i=1}^{n+1}4a_i^2\left(\int_{\partial\til{M}}dB_{(st_V(\delta),\theta)}(w)d\nu_{x_i}(\theta)\right)^2\right)^{1/2}\\
&\leq 2\left(\sum_{i=1}^{n+1}a_i^2\int_{\partial\til{M}} \Big(dB_{(st_V(\delta),\theta)}(w)\Big)^2 d\nu_{x_i}(\theta)\int_{\partial\til{M}}1\, d\nu_{x_i}\right)^{1/2}\\
&=2 \langle H_{\delta,V}(w),w\rangle^{1/2}
\end{split}
\end{equation}
via two applications of the Cauchy-Schwartz inequality.

For points $\delta\in\Delta_s^n$ where $st_V$ is nondegenerate, we now pick orthonormal
bases $\{u_1,\ldots ,u_n\}$ of $T_\delta(\Delta_s^n)$, and $\{v_1,\ldots ,v_n\}$ of
$T_{st_V(\delta)}(\til{M})$. We choose these so that $\{v_i\}_{i=1}^n$ are
eigenvectors of $H_{\delta,V}$, and $\{u_1,\ldots, u_n\}$ is the resulting basis obtained by
applying the orthonormalization process to the collection of pullback vectors
$\{(K_{\delta,V}\circ D_\delta(st_V))^{-1}(v_i)\}_{i=1}^n$. By the choice of bases, the
matrix $(\langle K_{\delta,V}\circ D_\delta(st_V)(u_i),v_j\rangle)$ is upper triangular, so
we obtain
\begin{align*}
|\det(K_{\delta,V})\cdot \Jac_\delta(st_V)| &=|\det(\langle K_{\delta,V}\circ
D_\delta(st_V)(u_i),v_j\rangle)|\\
& =|\prod_{i=1}^n\langle K_{\delta,V}\circ D_\delta(st_V)(u_i),v_i\rangle| \\
& \leq \prod_{i=1}^n 2 \langle H_{\delta,V}(v_i),v_i\rangle^{1/2} \\
&=2^n \det(H_{\delta,V})^{1/2}
\end{align*}
where the middle inequality is obtained via Equation (\ref{eqn:Q1-bounds-Q2}).
Hence we get the inequality
\begin{equation}\label{ineq:Jac}
|\Jac_\delta(st_V)|\leq 2^n\cdot \frac{\det(H_{\delta,V})^{1/2}}{\det(K_{\delta,V})}
\end{equation}
where
$$H_{\delta,V}=\int_{\partial\til{M}}\Big(dB_{(st_V(\delta),\theta)}\Big)^2 d\left(\sum_{i=1}^{n+1}a_i^2\nu_{x_i}\right)(\theta)$$
$$K_{\delta,V}=\int_{\partial\til{M}}DdB_{(st_V(\delta),\theta)}d\left(\sum_{i=1}^{n+1}a_i^2\nu_{x_i}\right)(\theta)$$
We remark that the above two positive semidefinite bilinear forms depend only on the point $st_V(\delta)\in \til{M}$ and the measure $\sum_{i=1}^{n+1}a_i^2\nu_{x_i}$ which arises from the coordinate of $\delta$ and the vertex set $V$. So it is convenient for us to denote
$$H_{x,\nu}=\int_{\partial\til{M}}\Big(dB_{(x,\theta)}\Big)^2 d\nu(\theta)$$
$$K_{x,\nu}=\int_{\partial\til{M}}DdB_{(x,\theta)}d\nu(\theta)$$
for any $x\in \til M$ and any probability measure $\nu$ on $\partial \til M$.

Thus we summarize the discussions in the following proposition.

\begin{prop}
Given $M$ a closed rank one manifold, if there exists a constant $C$ and a non-empty open set $U\subset M$ such that, for any probability measure $\nu$ fully supported on $\partial \til M$, and any $x\in \til M$ whose natural projection $p(x)\in U$, we have
$$\frac{\det(H_{x,\nu})^{1/2}}{\det(K_{x,\nu})}\leq C.$$
Then $||M||>0$.
\end{prop}

\begin{proof}
The proof simply follows from Theorem \ref{thm:straightening} and inequality (\ref{ineq:Jac}).
\end{proof}
\section{Proofs of Thereoms}

\begin{proof}[Proof of Theorem \ref{thm:general estimate}]

We use a similar idea to the proof of Theorem \ref{thm:straightening} except that we now pick as bump function the function $u^n(x)$ of Theorem \ref{thm:general estimate}. The theorem holds automatically when $u(x)\equiv 0$, so we can assume $u(p)>0$ for some $p\in M$. By Lemma \ref{lem:tr2-implies-rank1} and \ref{lem:strictly convex}, the barycentric straightening is well defined on the universal cover $\til{M}$. 

Let $\sum_{i=1}^la_i\sigma_i$ be any singular chain representing the fundamental class $[M]$ in the real coefficient, and $st(\sigma_i)$ be the barycentrically straightened simplex of $\sigma_i$. We have the following estimate
\begin{align}\label{equ:simpicial-lower-bound}
\int_Mu^n(x)dV &=\int_{[\sum a_i\sigma_i]}u^n(x)dV =\int_{[\sum a_ist(\sigma_i)]}u^n(x)dV \\
& \leq \sum_{i=1}^l |a_i|\cdot\left| \int_{st(\sigma_i)}u^n(x)dV\right| \\
& = \sum_{i=1}^l |a_i|\cdot\left| \int_{\Delta_s^n} u^n( st(\sigma_i)(\delta))|\Jac(st(\sigma_i))(\delta)|dV_s\right|  \label{eq:str-inequality}
\end{align}
If we denote $V_i$ the ordered vertex determined by $\sigma_i$, we obtain the following Jacobian estimate from inequality (\ref{ineq:Jac}),
\begin{equation}\label{eq:Jac-estimate}
|\Jac(st(\sigma_i))(\delta)|\leq 2^n\cdot \frac{\det(H_{\delta,V_i})^{1/2}}{\det(K_{\delta,V_i})}=2^n\cdot \frac{\det(H_{x,\nu})^{1/2}}{\det(K_{x,\nu})}
\end{equation}
where 
\begin{align*}
\langle
H_{x,\nu}(w),w\rangle_x&=\int_{\partial\til{M}}\Big(dB_{(x,\theta)}(w)\Big)^2 d\nu(\theta)\\
\langle
K_{x,\nu}(w),w\rangle_x&=\int_{\partial\til{M}}DdB_{(x,\theta)}(w,w)d\nu(\theta)
\end{align*}
with $x=st(\sigma_i)(\delta)$ and $\nu$ some probability measure fully supported on $\partial \til{M}$ depending on $\delta, V_i$. Although the dependency of $\nu$ is not specified, we will see in what follows that the estimate does not rely on $\nu$.

At the point $x$ where $u(x)>0$, we have matrix forms 
\begin{align*}
(dB_{(x,\theta)})^2&=\left[
   {\begin{array}{cc}
   1 & 0 \\
   0 & 0^{(n-1)}\\
\end{array} }
\right]\\
DdB_{(x,\theta)}&\geq u(x)\left[
   {\begin{array}{cc}
   0 & 0 \\
   0 & I^{(n-1)}\\
\end{array} }
\right]
\end{align*}
under a proper choice of orthonormal basis $e_1, e_2,..., e_n$ so that $e_1$ is the unit vector at $x$ pointing toward $\theta$. Therefore, we have for any $(x,\theta)$ with $u(x)>0$ that
$$(dB_{(x,\theta)})^2+\frac{1}{u(x)}DdB_{(x,\theta)}\geq I^{(n)}$$
Hence after integrating on any probability measure $\nu$, the following holds.
$$H_{x,\nu}+\frac{1}{u(x)}K_{x,\nu}\geq I^{(n)}$$
This inequality is independent on $\nu$. We now apply the following lemma from Besson-Courtois-Gallot \cite{Besson95} on $H_{x,\nu}$ and $\frac{1}{u(x)}K_{x,\nu}$.

\begin{lem} Let $H$ and $K$ be two $n\times n$ ($n\geq 3$) matrices, where $K$ is positive definite, and $H$ is positive semidefinite. If $H+K\geq I$ and $tr(H)=1$, then
$$\frac{\det(H)^{1/2}}{\det(K)}\leq \frac{n^{n/2}}{(n-1)^n}$$
\end{lem}

Thus, we obtain from inequality (\ref{equ:simpicial-lower-bound})-(\ref{eq:Jac-estimate}) the following
\begin{align*}
\int_Mu^n(x)dV &\leq 2^n \frac{n^{n/2}}{(n-1)^n} vol(\Delta_s^n)(\sum_{i=1}^l|a_i|)\\
&= \frac{n^{n/2}\omega_n}{2(n-1)^n}\sum_{i=1}^l|a_i|
\end{align*}
where $\omega_n$ is the volume of a unit $n$-sphere.

By taking the infimum over all fundamental class representatives $\sum a_i\sigma_i$, we have
$$||M||\geq \frac{2(n-1)^n}{n^{n/2}\omega_n}\int_M u^n(x)dV$$

\end{proof}

\begin{proof}[Proof of Theorem \ref{thm:dim3}]
We show $u(x)$ cannot be identically zero under the assumption of negative Ricci curvature, hence by Corollary \ref{cor:positive-under-rank-condition} the simplicial volume is positive.

Assume not, $u(x)=0$ for all $x\in M$, that is, there exists a higher $\op{rank}^+$ unit vector $v_x$ at every point $x$ on $M$. Since $\op{Ric}(v_x, v_x)$ is negative, $v_x$ can only be $\op{rank}^+$ two, that is, there exists a unique (up to sign) unit vector $u_x\perp v_x$ that is parallel along the geodesic ray formed by $v_x$. In particular, the plane spanned by $u_x$ and $v_x$ has curvature zero. Note that the zero curvature plane has to be unique among the Grassmanian two planes of $T_xM$, because if not, we can take any unit vector $w$ in the intersection of two zero curvature planes, and it is clear that $\op{Ric}(w,w)=0$ as $M$ is nonpositively curved, this contradicts the Ricci curvature condition. 

We denote on each tangent space $T_xM$ an orthonormal frame $e_1=v_x$, $e_2=u_x$, and $e_3$ orthogonal to the zero curvature plane. We may assume the zero curvature plane field is orientable by possibly passing to a suitable double cover of $M$. This gives a smooth vector field $e_3$ after a preferable choice of normal direction of the plane field, which at each point corresponds to the unique minimal Ricci curvature direction. (Note that $e_1, e_2$ might not be global vector fields.) We now apply the following Bochner's techique on $e_3$.

\begin{lem} (Bochner \cite{Bochner46}, see also \cite[Chapter 7-Exercise 6] {Peterson06}) 
If $X$ is a smooth vector field on a closed manifold $M$, then
$$\int_M \op{Ric}(X,X)dV=\int_M (\op{div} X)^2-\op{Tr}(\nabla X\circ \nabla X) dV$$
\end{lem}

If we set $X=e_3$ in the lemma above, we can compute $\op{div} X$ and $\op{Tr}(\nabla X\circ \nabla X)$ at each point using any choice of orthonormal basis. Indeed, if we fix a point $x\in M$, and pick an orthonormal frame $e_1$, $e_2$ and $e_3$ on $T^1M$ as above, we have
\begin{equation*}
\nabla_{e_1}e_3=0 \qquad\quad \nabla_{e_2}e_3=a_{21}e_1+a_{22}e_2 \quad\qquad \nabla_{e_3}e_3=a_{31}e_1+a_{32}e_2
\end{equation*}
as the $e_1$ and $e_2$ fields are parallel along the geodesic ray tangent to $e_1$, so is their orthogonal complement $e_3$. This implies
$$\op{div} X=<\nabla_{e_1}e_3,e_1>+<\nabla_{e_2}e_3,e_2>+<\nabla_{e_3}e_3,e_3>=a_{22}$$
and
\begin{align*}
\op{Tr}(\nabla X\circ \nabla X)&=<\nabla_{\nabla_{e_1}e_3}e_3,e_1>+<\nabla_{\nabla_{e_2}e_3}e_3,e_2>+<\nabla_{\nabla_{e_3}e_3}e_3,e_3>\\
&=0+a_{22}^2+0=a_{22}^2
\end{align*}
Hence the integrand of the right hand side of the identity in the above lemma is always $0$, but the left hand side is strictly negative. This gives a contradiction.
\end{proof}

\begin{proof}[Proof of Theorem \ref{thm:pointwise-k-ricci}]
We notice that the Ricci condition is an open condition, that is, there exists $\epsilon_0>0$ and an open set $U\subset M$, such that $\op{Ric}_{\floor{\frac{n}{4}}+1}(v_x, v_x) < -\epsilon_0$ for all $v_x\in T^1_xM$, and for all $x\in U$. We apply the following estimate.

\begin{thm}[\cite{CW17} Theorem 6] \label{thm:Jacobian estimate} Suppose $M$ is a closed non-positively curved
manifold with negative $(\floor{\frac{n}{4}}+1)$-Ricci curvature, and
$\til{M}$ is its Riemannian universal cover. Let $x\in \til{M}$, $\theta\in \partial
\til{M}$, and $\nu$ be any probability measure that has full support on $\partial
\til{M}$. Then
	there exists a universal constant $C$ that only depends on $(M,g)$, so that
	$$\frac{\det(\int_{\partial
	\til{M}}(dB_{(x,\theta)}(\cdot))^2 d\nu(\theta))^{1/2}}{\det(\int_{\partial
	\til{M}}DdB_{(x,\theta)}(\cdot,\cdot)d\nu(\theta))}\leq C$$
\end{thm}
Moreover, from the proof of Theorem 6 in \cite{CW17} we see that the above inequality is a pointwise estimate, that is, $C$ only depends on the dimension $n$, the norm of the curvature operators $\|R\|$, $\|\nabla R\|$, $\|\nabla^2 R\|$, and the corresponding $(\floor{\frac{n}{4}}+1)$-Ricci constant at $x$. Therefore, we see that the barycentric straightening is a straightening subordinated to $U$ under the curvature assumptions: the Jacobian estimate simply follows from inequality (\ref{ineq:Jac}) and the theorem above. Applying this to $\nu=\sum_{i=1}^{n+1}a_i^2\nu_{x_i}$, we have the following inequalities:
\begin{align*}
|\op{Jac}(st_n(f))(\delta)|&\leq 2^n\frac{\det(\int_{\partial
	\til{M}}(dB_{(x,\theta)}(\cdot))^2d\nu(\theta))^{1/2}}{\det(\int_{\partial
	\til{M}}DdB_{(x,\theta)}(\cdot,\cdot)d\nu(\theta))}\\
&\leq 2^nC
\end{align*}
whenever $st_n(f)(\delta)\in p^{-1}(U)$. Hence by Theorem \ref{thm:straightening}, the simplicial volume of $M$ is positive.

\end{proof}

\section{Concluding remarks}
In addition to their independent interest, both Corollary \ref{cor:positive-under-rank-condition} and Theorem \ref{thm:pointwise-k-ricci} represent progress towards Gromov's conjecture, where the former addresses conditions on the rank and the latter does on the curvature. The two conditions do not have a direct implication between them. Indeed, a rank one vector $v$ may have zero Ricci curvature, hence $\op{Ric}_k$ is also zero at $v$. On the other hand, a higher rank vector $v$ only sees a dimension two zero curvature plane through it, which is strictly weaker then $\op{Ric}_k(v,v)=0$ when $k>2$.

To summerize, if we denote by $\mathcal{M}^n$ the collection of all $n$-dimensional rank one manifolds that are not covered by our theorems (the higher rank case is clear by Theorem \ref{thm:rank-rigidity} and \cite{Lafont06b}), then we have for every $M\in \mathcal{M}^n$, $M$ satisfies:

\begin{enumerate}
	\item $M$ has nonpositive curvature, negative Ricci curvature, and is rank one.
	\item Every point on $M$ has a higher $\op{rank}^+$ vector.
	\item Every point on $M$ has a vector that has $\op{Ric}_{\floor{\frac{n}4}+1}=0$.
\end{enumerate}

It is worth pointing out that we are not aware of any such examples. Therefore one possible approach to verify Gromov's conjecture is showing $\mathcal{M}^n=\emptyset$. In fact, in Theorem \ref{thm:dim3}, we showed $\mathcal{M}^3=\emptyset$ by contradicting (a) with (b). However, our proof relies heavily on the low dimensional constraints that at each point the higher rank vectors give unique zero curvature planes. In higher dimensions, we believe that the dynamical properties of the geodesic flow of rank one manifolds (such as the distribution of regular/singular points at boundary at infinity) hold the key to providing a complete answer to this problem.

Motivated by these discussions, we offer the strengthened version of Gromov's conjecture.

\begin{conj}
	Let $M$ be a closed manifold admitting a metric of non-positive curvature with a point $x\in M$ of negative Ricci-curvature. Then $M$ has positive simplicial volume. 
\end{conj}

\def\polhk#1{\setbox0=\hbox{#1}{\ooalign{\hidewidth
  \lower1.5ex\hbox{`}\hidewidth\crcr\unhbox0}}} \def\cprime{$'$}
\providecommand{\bysame}{\leavevmode\hbox to3em{\hrulefill}\thinspace}
\providecommand{\MR}{\relax\ifhmode\unskip\space\fi MR }
\providecommand{\MRhref}[2]{%
  \href{http://www.ams.org/mathscinet-getitem?mr=#1}{#2}
}
\providecommand{\href}[2]{#2}

\end{document}